\documentclass[12pt]{article}
\usepackage{amsthm}
\usepackage{amsmath}
\usepackage{amsfonts}
\usepackage{microtype}

\theoremstyle{plain}
\newtheorem{prop}{Proposition}[section]
\newtheorem{conj}[prop]{Conjecture}
\newtheorem{thm}[prop]{Theorem}
\newtheorem{lem}[prop]{Lemma}
\newtheorem{cor}[prop]{Corollary}
\newtheorem{example}[prop]{Example}

\usepackage[T1]{fontenc}
\usepackage{libertine}
\usepackage[scaled=0.85]{beramono}%% or 0.82

\def\semicolon{;}
\def\applytolist#1{
    \expandafter\def\csname multi#1\endcsname##1{
        \def\multiack{##1}\ifx\multiack\semicolon
            \def\next{\relax}
        \else
            \csname #1\endcsname{##1}
            \def\next{\csname multi#1\endcsname}
        \fi
        \next}
    \csname multi#1\endcsname}

\def\calc#1{\expandafter\def\csname c#1\endcsname{{\mathcal #1}}}
\applytolist{calc}QWERTYUIOPLKJHGFDSAZXCVBNM;
\def\bbc#1{\expandafter\def\csname bb#1\endcsname{{\mathbb #1}}}
\applytolist{bbc}QWERTYUIOPLKJHGFDSAZXCVBNM;
\def\bfc#1{\expandafter\def\csname bf#1\endcsname{{\mathbf #1}}}
\applytolist{bfc}QWERTYUIOPLKJHGFDSAZXCVBNM;

\newcommand{\fsl}{\mathfrak{sl}_2}

\newcommand{\complexzero}{0\rightarrow\mathbb Q [\beta]\rightarrow 0}
\newcommand{\complexk}[1]{0\rightarrow\mathbb Q [\beta]\xrightarrow{\beta^{#1}}{}\mathbb Q [\beta]\rightarrow 0}
\newcommand{\red}[1]{{\color{red}#1}}

\newcommand{\mmain}[1]{{\color{blue}\textbf{\tt #1}}}

\newcommand{\arxiv}[1]{\href{http://arxiv.org/abs/#1}{\tt arXiv:\nolinkurl{#1}}}
\newcommand{\doi}[1]{\href{http://dx.doi.org/#1}{{\tt DOI:#1}}}

\usepackage{xcolor}
\definecolor{dark-red}{rgb}{0.7,0.25,0.25}
\definecolor{dark-blue}{rgb}{0.15,0.15,0.55}
\definecolor{medium-blue}{rgb}{0,0,0.65}

\usepackage[pdftex,plainpages=false,hypertexnames=false,pdfpagelabels,breaklinks]{hyperref}
\hypersetup{
   colorlinks, linkcolor={purple},
   citecolor={medium-blue}, urlcolor={medium-blue}
}

\DeclareMathOperator{\tr}{tr}

\title{Computing annular Khovanov homology}
\author{Hilary Hunt, Hannah Keese, Anthony Licata, Scott Morrison \\ Australian National University}

\begin{document}
\maketitle

\begin{abstract}
We define a third grading on Khovanov homology, which is an invariant of annular links but changes by $\pm 1$ under stabilization. We illustrate the use of our computer implementation, and give some example calculations.
\end{abstract}

Annular Khovanov homology was first defined in \cite{MR2113902}, as a special case of a construction for a link homology theory on $I$-bundles over surfaces. It has been investigated in a number of articles (e.g. \cite{MR2728482, MR3147412, MR2866927, 1212.2222, 1303.1986, 1305.2183}), and \cite{GLW} begins exploring the representation-theoretic context, giving annular Khovanov homology the structure of a representation of the $\mathfrak{sl}_2$ current algebra (see also \cite{KeeseThesis, QR}). The purpose of this paper is to discuss a closely related annular link invariant, which recovers both annular Khovanov homology and the usual Khovanov homology as specializations.  While much of this mathematics is well known to people working in Khovanov homology, not all appears explicitly in the literature.
In order to explain our computer implementation we've included some background material on homological algebra to make the description self-contained.

\section{The annular grading on Khovanov homology}
Given a point in the complement of a link diagram, the chain groups in the Khovanov complex can be given a third `annular' grading, as follows.

In each complete resolution of a link diagram, three types of cycles are distinguished: cycles that do not enclose the marked point, called trivial cycles, and cycles that enclose the marked point, called either even or odd nontrivial cycles, which are distinguished by the parity of their nesting depth.

Each cycle is assigned a 2-dimensional vector space spanned by symbols $\{v_+, v_-\}$ with Euler gradings $+1$ and $-1$. When the cycle is trivial, both symbols have annular grading zero. When the cycle is non-trivial, of either parity, $v_+$ has annular grading $+1$ and $v_-$ has annular grading $-1$.

The usual Khovanov differential can be decomposed as $d_{Kh}=d_{Kh}^0+d_{Kh}^{-2}$, where $d_{Kh}^0$ preserves the annular grading and $d_{Kh}^{-2}$ shifts the annular grading down by two. The  sutured annular Khovanov homology of a link diagram is just defined as the homology with respect to the differential $d_{Kh}^0$.

(In fact, the vector spaces associated to cycles can be given the structure of $\fsl$ representations: trivial cycles are assigned $V_0 \oplus V_0$, while odd non-trivial cycles are assigned the standard representation $V_1$, and even non-trivial cycles are assigned the dual representation $V_1^*$. The annular grading then agrees with the $\fsl$ weight space grading. The $\fsl$ action only commutes with the $d^{Kh}_0$ term of the differential, so descends to the sutured annular Khovanov homology, but not to any of the other homologies discussed here.)

We're also interested in the degeneration of sutured annular Khovanov homology to the usual Khovanov homology. Rather than working with spectral sequences, we introduce a parameter $\beta$ to control the grading shifts, and work over the ring $\bbQ[\beta]$. We give the parameter $\beta$ the annular grading $+2$. 

\begin{lem}
The differential $\tilde{d}=d_{Kh}^0 + \beta d_{Kh}^{-2}$ gives a graded complex of free $\bbQ[\beta]$ modules whose homotopy type is an annular link invariant.
\end{lem}

\begin{proof}
Let $f$ be a chain map for an annular Reidemeister move (that is, a Reidemeister move performing an isotopy that stays away from the puncture).  Then $f$ commutes with the Khovanov differential $d_{Kh}$.  Furthermore, since $f$ and $d_{Kh}^0$ preserve the annular grading, and $d_{Kh}^{-2}$ is in degree $-2$, $f$ must commute with each term in the decomposition of $d_{Kh}$ separately: $d_{Kh}^0 f= f d_{Kh}^0$ and $d_{Kh}^{-2} f = f d_{Kh}^{-2}$.  Thus $f$ commutes with $\tilde{d}$.
\end{proof}

Further, it is easy to see that under the isotopy that moves a single strand across the puncture, each homogeneous representative of sutured annular Khovanov homology changes annular grading by $\pm 1$.

The invariant we have defined is an up-to-homotopy complex in the category of free $\bbQ[\beta]$ graded modules, with differentials all in degree 0. We now describe in more detail exactly what such a complex looks like.
The objects are just direct sums of $\bbQ[\beta]$ in different gradings, and differentials are matrices with rows and columns indexed by these direct sums. Each matrix entry is then a degree 0 map from $\bbQ[\beta]\{n\}$ to $\bbQ[\beta]\{m\}$ for some integers $n$ and $m$. The only such maps are $z \beta^{n-m}$ if $n-m \geq 0$, for some $z \in \bbQ$, and zero otherwise.

In the additive category of complexes we have:

\begin{lem}
A complex of free $\bbQ[\beta]$ graded modules is isomorphic to a direct sum of shifted copies of
\[
\complexzero
\]
and
\[
\complexk{k}
\]
for some $k \geq 0$.

\end{lem}
\begin{proof}
If all the differentials are zero, we are done.

Otherwise, one of the differentials, say $d_k$, has a non-zero entry. In particular, it has a lowest degree non-zero entry. We may assume (by permuting direct sums) that this is the top left entry, and so we have

\[
d_k=\left(
\begin{array}{c|c}
a\beta^n&\gamma\\ \hline
\kappa &\omega
\end{array}
\right)
\]
where $\beta^n$ divides $\gamma,\kappa,$ and $\omega$.
Following the argument in \cite[Lemma 4.2]{math.GT/0606318} 
we see that 
\[
\cdots \to [A]
\xrightarrow{\left(\begin{array}{c}\varphi_1 \\ \varphi_2\end{array}\right)}
\left[\begin{array}{c}B \\ C\end{array}\right]
\xrightarrow{\left(
\begin{array}{cc}
a\beta^n&\gamma\\
\kappa &\omega
\end{array}
\right)}
\left[\begin{array}{c}D \\ E\end{array}\right]
\xrightarrow{\left(\begin{array}{cc}\psi_1 & \psi_2\end{array}\right)}
[F] \to \cdots
\]
is isomorphic to the direct sum 
\[
\cdots \to [A]
\xrightarrow{\left(\begin{array}{c}0 \\ \varphi_2\end{array}\right)}
\left[\begin{array}{c}B \\ C\end{array}\right]
\xrightarrow{\left(
\begin{array}{cc}
a\beta^n&0\\
0 &\omega-\kappa (a\beta^{n})^{-1} \gamma
\end{array}
\right)}
\left[\begin{array}{c}D \\ E\end{array}\right]
\xrightarrow{\left(\begin{array}{cc}0 & \psi_2\end{array}\right)}
[F] \to \cdots
\]
even when $a\beta^n$ is not an isomorphism. Here by $\kappa (a\beta^{n})^{-1} \gamma$ we mean $a^{-1} \kappa \gamma'$, where $\gamma = \beta^n \gamma'$.  Hence, we can break down the whole complex into a direct sum of shifted copies of
\[
\complexzero
\]
and
\[
\complexk{k}.
\]
(Recall here that `shifted' refers to all of Euler gradings, homological gradings, and annular gradings.)%
\end{proof}

When we pass to the homotopy category we have:

\begin{cor}
\label{cor:he}
Such a complex is homotopy equivalent to a direct sum of shifted copies of
\[
\complexzero
\] 
and
\[
\complexk{k}
\]
for some $k > 0$.

\end{cor}
\begin{proof}
The map $\beta^0$ is invertible. This makes every copy of 
\[
\complexk{0}
\]
contractible. Hence the complex is homotopy equivalent to the complex where every copy of this summand is removed. We are left with a direct sum of shifted copies of
\[
\complexzero
\]
and
\[
\complexk{k}
\]
for some $k > 0$.
\end{proof}

In fact, the up-to-homotopy representative above is unique.
\begin{lem}
Suppose that 
$$ \mathcal C = W_0 \otimes \left(\complexzero\right) \oplus \bigoplus_{i \geq 1} W_i \otimes \left(\complexk{i}\right) $$
and 
$$ \mathcal C' = W_0' \otimes \left(\complexzero\right) \oplus \bigoplus_{i \geq 1} W_i' \otimes \left(\complexk{i}\right) $$
are homotopic, where the $W_i$ and $W_i'$ are triply graded vector spaces.
Then $W_i \cong W_i'$ for each $i \geq 0$.
\end{lem}
\begin{proof}
First, $H(\mathcal C, \mathbb Q[\beta]/(\beta - 1)) \cong W_0$, so we have $W_0 \cong W_0'$.

Consider adding the relation $\beta^i = 0$. Each $\complexk{k}$ summand then has homology
\begin{align*}
\Big((z^{2(\min(i-k,0))} + z^{2(\min(i-k,0)+1)} & + ... + z^{2(i-1)}) \;+  \\
 &  + \;t(1 + z^2 + ... + z^{2(\min(k-1,i-1))})\Big)\mathbb Q.
\end{align*}
If we ignore the annular grading this is just
$(1+t) \mathbb Q^{i-1}$ when $i \leq k$, and $(1+t) \mathbb Q^k$ when $i > k$.

Thus formally we have
\begin{align*}
H(\mathcal C, \mathbb Q[\beta]/\beta^{n+1}) - H(\mathcal C, \mathbb Q[\beta]/\beta^n) & = W_0 \oplus (1+t) \bigoplus_{k \geq n} W_k
\end{align*}
and
\begin{align*}
- H(\mathcal C, \mathbb Q[\beta]/\beta^{n+2}) + 2 H(\mathcal C, \mathbb Q[\beta]/\beta^{n+1}) -  H(\mathcal C, \mathbb Q[\beta]/\beta^{n}) & = (1+t) W_n 
\end{align*}
This allows us
to compute the graded dimensions of the $W_k$, as homotopy invariants.
\end{proof}

\begin{thm}
There is a collection of triply graded vector spaces over $\bbQ$, $\{W_i\}_{i \geq 0}$, which are invariants of an annular link.
\begin{itemize}
\item Forgetting the annular grading of $W_0$ recovers the usual Khovanov homology of the link.
\item Taking the direct sum $$W_0 \oplus \bigoplus_{i \geq 1} (\mathbb Q \oplus t \mathbb Q) \otimes W_i$$ recovers the triply graded sutured annular Khovanov homology.
\end{itemize}
\end{thm}
\begin{proof}
Write the triply graded complex $CKh(L)$ of $\bbQ[\beta]$ modules uniquely as 
\begin{align*}
W_0 \otimes & \left(\complexzero\right) \oplus \\
            & \bigoplus_{i > 0} W_i \otimes \left( \complexk{i} \right).
\end{align*}
With $\beta = 1$ (usual Khovanov homology), $\complexk{i}$ becomes contractible, while with $\beta = 0$ (sutured annular Khovanov homology), $\complexk{i}$ becomes $\mathbb Q \oplus t \mathbb Q$.
\end{proof}

In fact, this up-to-homotopy representative encodes all the information usually described by the spectral sequence converging from the annular Khovanov homology to the usual Khovanov homology. The successive pages of the spectral sequence are given by 
$E_{k+1} = H(E_k, d_k)$, where as a vector space $E_0$ is the complex $CKh(L)$, and the differentials $d_k$ are the coefficients of the $\beta$ expansion of the differential on $CKh(L)$. We find that
$$E_j \cong W_0 \oplus (1+t) \bigoplus_{k > j/2} W_k.$$
To see this, 
 it suffices to consider the spectral sequence when the chain complex is a single summand $\complexk{k}$. (Note that the homotopy equivalence over $\mathbb Q[\beta]$ described in Corollary \ref{cor:he} specializes at $\beta = 1$ to a homotopy equivalence over $\mathbb Q$.) This complex specializes at $\beta = 1$ to $0 \rightarrow \mathbb Q \xrightarrow{1} \mathbb Q\{2k\} \rightarrow 0$, so we have $$d = \sum_{j \geq 0} d_k = d_{2k} = 1.$$ Thus the $E_0$ page is just $(1+t)\mathbb Q$, and this survives through to the $2k$-th page, on which the induced differential appears, killing everything thereafter.

\section{Implementation}
We have implemented the invariant described above, in the special case of annular links which are presented as braid closures (with the braid closure happening around the puncture in the annulus; this restriction could be relaxed relatively easily).

To use this implementation, you will first need to install a copy of the {\tt KnotTheory`} mathematica package which is available at
{\url{http://katlas.org/wiki/The_Mathematica_Package_KnotTheory`}}. You will also need the Mathematica notebook {\tt KhBraids.nb}, available 
in the {\tt arXiv} sources of this note, or from
the {\tt github} repository { \url{https://github.com/semorrison/KhBraids/}}. After running all the cells in that notebook, you will have a number of functions available, which are illustrated by the following examples.

Lines typeset in blue represent input you might type into Mathematica; the following line in black represents the output you might receive.

One can of course calculate the Khovanov homology of a braid closure; the output is the Poincar\'e polynomial, with $q$ tracking the Euler grading and $t$ the homological grading. 
\vspace{2mm}
\begin{samepage}
\hrule\vspace{2mm}
\noindent\mmain{Kh[BR[Knot["3\_1"]]]}
\begin{verbatim}
1/q^3+1/q+1/(q^9 t^3)+1/(q^5 t^2)
\end{verbatim}
\hrule
\end{samepage}
\vspace{3mm}
To calculate sutured annular Khovanov homology, we use the optional second parameter, which allows specifying a differential. The differentials \mmain{AnnularDifferential} and \mmain{LeeDifferential} are defined by the package; you can define others yourself.
\vspace{2mm}\hrule\vspace{2mm}
\noindent\mmain{Kh[BR[Knot["3\_1"]], AnnularDifferential]}
\begin{verbatim}
1/q^5+1/q^3+1/q+1/(q^9 t^3)+1/(q^5 t^2)+1/(q^5 t)
\end{verbatim}
\hrule\vspace{3mm}
A different command gives the decomposition of sutured annular Khovanov homology into $\mathfrak{sl_2}$ representations. Here $V[i]$ denotes the irreducible representation with highest weight $i$. 
\vspace{2mm}\hrule\vspace{2mm}
\noindent\mmain{AnnularKh[BR[Knot["3\_1"]]]}
\begin{verbatim}
V[0]/(q^9 t^3)+V[0]/(q^5 t^2)+V[0]/(q^5 t)+V[2]/q^3
\end{verbatim}
\hrule\vspace{3mm}
The invariant over $\mathbb{Q}[\beta]$ described above can also be calculated using
the command \mmain{SpectralAnnularKh}; an example appears in the next section.
Finally, one can ask for explicit representatives of Khovanov homology (with respect to your preferred differential).
\begin{samepage}
\vspace{2mm}\hrule\vspace{2mm}
\noindent\mmain{KhRepresentativesInGrading[BR[Knot["3\_1"]], -1, -5, AnnularDifferential]}
\begin{verbatim}
{{1,1,1}}
\end{verbatim}
\hrule\vspace{3mm}
\end{samepage}
These are given with respect to an ordered basis which can be viewed via the command \mmain{AllBasisVectorsInGrading}. The particulars of the notation are left as an exercise for the reader! 
\vspace{2mm}\hrule\vspace{2mm}
\noindent\mmain{AllBasisVectorsInGrading[BR[Knot["3\_1"]], -1, -5]}
\begin{verbatim}
{
  BasisVector[Resolution[BR[2,{-1,-1,-1}],{0,0,1},{{2,4,6,8,9,11,13,14}}],
    {{2,4,6,8,9,11,13,14}->-1}],
  BasisVector[Resolution[BR[2,{-1,-1,-1}],{0,1,0},{{2,4,5,7,10,12,13,14}}],
    {{2,4,5,7,10,12,13,14}->-1}],
  BasisVector[Resolution[BR[2,{-1,-1,-1}],{1,0,0},{{1,3,6,8,10,12,13,14}}],
    {{1,3,6,8,10,12,13,14}->-1}]
}
\end{verbatim}
\mmain{CubeOfResolutions[BR[Knot["3\_1"]]] //TableForm}
\begin{verbatim}
Resolution[BR[2,{-1,-1,-1}],{0,0,0},{{2,6,10,13},{4,8,12,14}}]
Resolution[BR[2,{-1,-1,-1}],{0,0,1},{{2,4,6,8,9,11,13,14}}]
Resolution[BR[2,{-1,-1,-1}],{0,1,0},{{2,4,5,7,10,12,13,14}}]
Resolution[BR[2,{-1,-1,-1}],{0,1,1},{{7,9},{2,4,5,11,13,14}}]
Resolution[BR[2,{-1,-1,-1}],{1,0,0},{{1,3,6,8,10,12,13,14}}]
Resolution[BR[2,{-1,-1,-1}],{1,0,1},{{1,11,13,14},{3,6,8,9}}]
Resolution[BR[2,{-1,-1,-1}],{1,1,0},{{3,5},{1,7,10,12,13,14}}]
Resolution[BR[2,{-1,-1,-1}],{1,1,1},{{3,5},{7,9},{1,11,13,14}}]
\end{verbatim}
\hrule\vspace{3mm}

\section{Examples}
As an illustration of the strength of the invariant described here, we now compute the triply graded invariants of $\tr{b_1}$ and $\tr{b_2}$, where $b_1$ and $b_2$ are non-conjugate 5-strand braids, whose trace closures are both the $8_{12}$ knot from the Rolfsen tables. In fact, these braids come from Vaughan Jones' table of braid representatives from \cite{MR0908150} and from Thomas Gittings' table of minimal braids in \cite{math.GT/0401051}.

\begin{example}
Below we show two computations, only indicating the coefficient of $E$ (the triply graded dimension of $W_0$ above). At $z=1$ both recover the usual Poincar\'e polynomial of Khovanov homology. For convenience, terms with differing powers of $z$ are highlighted in red.

\vspace{2mm}\hrule\vspace{2mm}
\noindent\mmain{
SpectralAnnularKh[BR[5,\{-1,2,-1,-3,2,4,-3,4\}]
}
\begin{align*}
\Big(t^4(q^9 z) & + t^3(q^5 \red{z^{-1}} + q^7 z) + t^2 (q^3 z^{-1} + 3 q^5 z) + t(3q z^{-1} + 2 q^3 z) + \\
           & + (3q z +3 q^{-1}z^{-1}) + t^{-1}(2q^{-3}z^{-1}+ 3 q^{-1} z) + t^{-2}(3 q^{-5}z^{-1} + q^{-3}z) + \\
           & + t^{-3}(q^{-7}z^{-1} + q^{-5}\red{z}) + t^{-4}q^{-9}z^{-1}\Big) E \\
           & + \Big( \ldots \Big) C[1]
\end{align*}

\noindent\mmain{
SpectralAnnularKh[BR[5,\{-1,2,-3,4,-3,4,-2,-1,3,2\}]
}
\begin{align*}
\Big(t^4(q^9 z^3) & + t^3(q^5 \red{z} + q^7 z) + t^2 (q^3 z^{-1} + 3 q^5 z) + t(3q z^{-1} + 2 q^3 z) + \\
           & + (3q z +3 q^{-1}z^{-1}) + t^{-1}(2q^{-3}z^{-1}+ 3 q^{-1} z) + t^{-2}(3 q^{-5}z^{-1} + q^{-3}z) + \\
           & + t^{-3}(q^{-7}z^{-1} + q^{-5}\red{z^{-1}}) + t^{-4}q^{-9}z^{-3}\Big) E \\
           & + \Big( \ldots \Big) C[1]
\end{align*}
\hrule\vspace{3mm}

This calculation shows that the closures are not isotopic as annular links,
so the braids are non-conjugate, and an isotopy of the closures as links in $S^3$ requires at least one stabilisation.
\end{example}

Although we haven't shown them, the graded dimensions, and indeed just the dimensions, of $W_1$ (that is, the coefficient of $C[1]$ above) are in fact also different.

We next make some observations about the annular gradings of stablised unknots.
Our package provides the function \mmain{Stabilization}, used for example as 
\vspace{2mm}\hrule\vspace{2mm}
\noindent\mmain{Stabilization[BR[2, \{-1, -1, -1\}], \{-1, 1\}]]}
\begin{align*}
BR[4, {-1, -1, -1, -2, 3}]
\end{align*}
\hrule\vspace{3mm}

For example, we have
\begin{samepage}
\vspace{2mm}\hrule\vspace{2mm}
\noindent\mmain{SpectralAnnularKh[Stabilization[BR[1, \{\}], \{-1, -1, 1, -1\}]]}
\begin{align*}
(q z + q^3 z^3) E + ( \ldots ) C[1]
\end{align*}
\hrule
\end{samepage}
\vspace{3mm}

The spread of annular gradings in annular Khovanov homology is simply the spread of $\mathfrak{sl}_2$ gradings in the largest irreducible representation of $\mathfrak{sl}_2$ appearing, which is determined solely by the number strands in the braid \cite{GLW}.

On the other hand, the annular spread of Khovanov homology may be more interesting. For many small examples, it is preserved under stabilization, but this is not always the case:
\vspace{2mm}\hrule\vspace{2mm}
\noindent\mmain{SpectralAnnularKh[BR[3, \{-1, 2, -1, 2\}]]}
\begin{align*}
\left(q^5 t^2 z+\frac{1}{q^5 t^2 z}+\frac{q t}{z}+\frac{z}{q t}+q z+\frac{1}{q z}\right) E + ( \ldots ) C[1]
\end{align*}
\noindent\mmain{SpectralAnnularKh[Stabilization[BR[3, \{-1, 2, -1, 2\}], \{1\}]]}
\begin{align*}
\left(q^5 t^2+\frac{1}{q^5 t^2 z^2}+q t+\frac{1}{q t}+q+\frac{1}{q z^2}\right) E + ( \ldots ) C[1]
\end{align*}
\noindent\mmain{SpectralAnnularKh[Stabilization[BR[3, \{-1, 2, -1, 2\}], \{1,1\}]]}
\begin{align*}
\left(q^5 t^2 z+\frac{1}{q^5 t^2 z^3}+\frac{q t}{z}+\frac{1}{q t z}+\frac{q}{z}+\frac{1}{q z^3}\right) E + ( \ldots ) C[1]
\end{align*}
\hrule\vspace{3mm}
Thus the figure-eight knot has annular spread of 2, as does its positive stabilization, while its double positive stabilization has annular spread of 4.

Generally, the behavior of the annular grading under stabilization seems quite complicated.

\section{Conjectures}

We conclude with a few conjectures.

\begin{conj}
The stabilized unknot $\sigma_n^{\epsilon_n} \sigma_{n-1}^{\epsilon_{n-1}} \cdots \sigma_1^{\epsilon_1}$, where $\epsilon_i = \pm 1$ has
$\dim W_0  = z^{- \sum \epsilon_i} (q z + q^{-1} z^{-1})$.
\end{conj}

\begin{conj}
The spectral sequence from sutured annular Khovanov homology to Khovanov homology collapses immediately. In the notation above, all the invariants $W_k$ for $k \geq 2$ vanish.
\end{conj}

It would be interesting to find a pair of non-conjugate braids, neither of which admit destabilizations, but with isotopic closures and with summands in $W_0$ differing in annular grading by at least 4. This would show that any isotopy requires at least two stabilizations.

\bibliographystyle{alpha}
\bibliography{bibliography}
\end{document}